\documentclass{amsart}
\usepackage[OT2,T1]{fontenc}
\usepackage{amsmath,amssymb,color,xcolor,graphicx,mathtools}
\usepackage{colonequals}
\usepackage{mathrsfs}
\usepackage{multicol}
\usepackage{nicefrac}
\usepackage[inline, shortlabels]{enumitem}
\usepackage{hyperref}
\hypersetup{colorlinks=true,urlcolor=blue,citecolor=blue,linkcolor=blue}
\usepackage{comment}
\usepackage{subcaption}
\usepackage[strings]{underscore}
\usepackage{tikz-cd}
\usepackage[nameinlink]{cleveref}
\Crefname{section}{\S\!\!}{\S\!\!}
\usepackage{anyfontsize}
\usepackage{lscape}
\usepackage[ruled,vlined]{algorithm2e}
\usepackage{rotating}
\usepackage{biblatex}
\newcommand{\myKer}{\mathop{\rm ker}}
\newcommand{\myIm}{\mathop{\rm Im}}

\newcommand{\rank}{\mathop{\rm rk}}
\newcommand{\Spec}{\mathop{\rm Spec}}

\numberwithin{equation}{section}

\theoremstyle{plain}
\newtheorem{lemma}[equation]{Lemma}

\newtheorem{thm}[equation]{Theorem}

\theoremstyle{remark}

\newtheorem{rmk}[equation]{Remark}
\newtheorem{example}[equation]{Example}

\theoremstyle{definition}
\newtheorem{definition}[equation]{Definition}

\DeclareSymbolFont{cyrletters}{OT2}{wncyr}{m}{n}
\DeclareMathSymbol{\Sha}{\mathalpha}{cyrletters}{"58}

\newcommand{\F}{\mathbb F}

\newcommand{\PP}{\mathbb P}

\newcommand{\Q}{\mathbb Q}

\newcommand{\Z}{\mathbb Z}

\newcommand{\bbQ}{\mathbb{Q}}

\newcommand{\bbZ}{\mathbb{Z}}

\newcommand{\calO}{\mathcal O}

\newcommand{\Magma}{\textsf{Magma}\/}

\newcommand{\isom}{\simeq}

\newcommand{\tensor}{\otimes} %
\DeclareMathOperator{\Aut}{Aut}

\DeclareMathOperator{\End}{End}

\DeclareMathOperator{\Galois}{Gal}

\DeclareMathOperator{\ord}{ord}

\newcommand{\blair}[1]{{\color{blue} \textsf{[Blair: #1]}}}

\usepackage{xstring}

\title{Arithmetic Information of Rational Elliptic Surfaces, and Shioda's Rank 68 Surface}
\author{Blair Butler}

\author{Andreas-Stephan Elsenhans}

\date{\today}

\begin{document}
\begin{abstract}
    The field of definition of the Mordell-Weil group of an elliptic surface $E \rightarrow \mathbb{P}^1$ defined over $\Q$ is the smallest number field $k$ such that all of its $\Bar{\Q}(t)$-rational points are defined over $k(t)$. In this paper, we present an algorithm, implemented in \Magma{}, which can determine arithmetic information, including the field of definition, associated to any rational elliptic surface. As an application of this, we also demonstrate that the field of definition of Shioda's rank $68$ elliptic surface given by $y^2 = x^3 + t^{360} + 1$ is a number field of degree $829,440$.
\end{abstract}

\maketitle

\linespread{1.3}

\section{Introduction and Main Results}
Let $k_0$ be a field, and $C$ a smooth projective curve defined over $k_0$. An elliptic surface is a smooth projective surface $S$, which is given with a relatively minimal elliptic fibration $f: S \rightarrow C$. We also require the condition that $f$ admits a section. This is commonly referred to as a Jacobian elliptic surface, though here we will simply call these elliptic surfaces.
These are in $1$-$1$ correspondence with elliptic curves $E$ over a function field $k_0(C)$. 
If we take the base curve $C$ to be $\PP^1$, then our elliptic surface admits a globally minimal Weierstrass equation of the form
\begin{equation}\label{Rational Weiestrass Form}
y^2 + a_1(t)xy + a_3(t)y = x^3 + a_2(t)x^2 + a_4(t)x + a_6(t)     
\end{equation}

where $a_i(t) \in k_0(t)$. An elliptic surface is called a rational elliptic surface if it is birational to $\PP^2$, the projective space of dimension two. In this case we have $\deg a_i(t) \leq i$ for all $i$. The surface is a $K3$ elliptic surface if $\deg a_i(t) \leq 2i$ for all $i$, and there exists at least one $i$ such that $\deg a_i(t) > i$, otherwise we call it an honestly elliptic surface. We call the number field $k_0$ the field of definition for the  elliptic surface. We call the set of sections $E(\overline{\Q}(t))$ the Mordell-Weil group of the elliptic surface. So long as the elliptic surface is non-constant on $\overline{\Q}(t)$, the Mordell-Weil group is a finitely generated abelian group \cite[Theorem III.6.1]{silverman2013advanced}. In addition, since the coefficients involved in the coordinates of generators of $E(\overline{\Q}(t))$ are algebraic, we can determine the smallest number field $k$ which contains them all. We call this $k$ the Field of Definition (of the Mordell-Weil group). In other literature, this is also referred to as the splitting field. 
    In order to determine the field of definition, we can turn to Galois representations. Let $G = \Galois(\overline{k_0}/k_0)$ denote the Galois Group. Then the natural action of $G$ on $E(\overline{k_0}(t))$ preserves the height pairing $\langle \cdot,\cdot\rangle$ and hence it gives a Galois representation $$\rho : G \rightarrow \Aut(E(\overline{\Q}(t))/E(\overline{\Q}(t))_{\mathrm{tor}}, \langle \cdot,\cdot\rangle).$$
    The target group is a finite group since the height pairing is positive definite up to torsion, and so is the image group $\myIm(\rho)$. Now the field of definition $k$ is precisely the Galois extension of $k_0$ corresponding to $\myKer(\rho)$ in the sense of Galois theory, and we have $\Galois(k/k_0) = \myIm(\rho)$. For further details, see \cite[Chapter 9]{schutt2019mordell}.

An important method to determine the rank of an elliptic surface is through the Shioda-Tate formula. This gives a relationship between the independent sections and the Néron–Severi lattice and the configuration of singular fibres:
$$\rank  E(\overline{\mathbb{Q}}(t)) = \rho -2 - \sum_{v\in \mathbb{P}^1}(m_v - 1),$$
where $\rho$ is the Picard number of the elliptic surface $E$.
The singular fibres can be determined explicitly using Tate's algorithm. Determining the rank of the Néron–Severi lattice is a non-trivial task, and can typically only be done in a relatively small class of elliptic surfaces.

The first is the case of rational elliptic surfaces. In this case, the rank of the Néron–Severi lattice is always equal to $10$, and so the rank of a rational elliptic surface can be determined just by the reduction types of the singular fibres. From there, finding the independent generators, as well as the field of definition is a relatively involved algorithm, which has been implemented in \Magma.
\begin{thm}\label{Rational Elliptic SUrface Theorem}
There is an algorithm that determines the rank, splitting field, automorphism group of the field of definition, generating sections, torsion sections, and irreducible representations of any rational elliptic surface defined over a number field $k_0$ in minimal Weierstrass form.
\end{thm}
 Another case is when the elliptic surface is also a Delsarte surface. Such a surface is
    birational to a hypersurface in $\mathbb{P}^3$
 defined by a polynomial with exactly 4
 terms. Elliptic Delsarte surfaces have previously been fully determined to consist of $42$ distinct families, and determining the generators and field of definition of a specific Delsarte elliptic surface is seemingly equivalent to determining the same information for an appropriate selection of rational and $K3$ elliptic surfaces. We focus on the elliptic surface 
 \begin{equation}\label{Shioda Rank 68 Eliptic Surface}
 E_{68}: \quad  y^2 = x^3 + t^{360} + 1.    
 \end{equation}
 This was first given by Shioda, and has rank $68$. Currently, this is the elliptic surface with the highest known rank, although there are other families of elliptic surfaces with rank $68$ as well, though these are $\overline{\mathbb{Q}}$-isomorphic to Shioda's surface \cite[Section 4.1]{heijne2011elliptic}. By Silverman's specialization \cite[Theorem III.11.4]{silverman2013advanced}, we can use this to construct an infinite family of elliptic curves with rank at least $68$, and so there is great interest to determine the field of definition for this surface.
We proceed in a similar way to \cite{chahal1999sections} to decompose our elliptic surface into 11 "smaller" elliptic surfaces, and analyse these separately. Putting all of this together, we will see:
\begin{thm}\label{Rank 68 Surface Theorem}
    The field of definition of the Mordell-Weil group of the rank $68$ elliptic surface $$y^2 = x^3 + t^{360} + 1$$
    is a number field of degree $829,440 = 2^{11}\cdot 3^4 \cdot 5$.
\end{thm}
This elliptic surface has been widely explored, for instance \cite{bootsma2023certain} explores this surface over fields of finite characteristic. Here the author gave the field of definition of the elliptic surface in terms of an explicit set of polynomials, and then determined a prime number such that all $68$ sections are defined over $\F_p(t)$. Furthermore, the field of definition of this elliptic surface has also been explored in \cite{salami2025splitting}, who was able to give an explicit description of the field of definition as a compositum of polynomials, but was unable to determine the degree of the number field.

For rational elliptic surfaces over finite fields, \Magma\ functions exist to determine the Mordell-Weil group and generators, and were implemented by Jasper Scholten in $2005$. Extending this implementation was not previously possible due to the difficulty in the construction of the automorphism group, and was not effective for elliptic surfaces whose field of definition have degree higher than say $50$. As a result, this required further algorithms to be developed which allowed us to compute the automorphism groups for number fields of higher degree, and in this article we compute automorphism groups of number fields of degree up to $240$.

All computations were done on a 2024 MacBook Pro, with 16GB RAM, running macOS Tahoe 26.5, using \Magma\ version V2.29-1.
\subsection{Outline of the paper}
In Section \ref{Section:Steph}, we show how to use the Galois group package of \Magma\
to explicitly construct the splitting field and its automorphism group of a polynomial. In Section \ref{Section: Rational}, we prove Theorem \ref{Rational Elliptic SUrface Theorem}, as well as giving an overview and some examples of the algorithm to completely determine all of this arithmetic information associated to a rational elliptic surface in \Magma. Finally, in Section \ref{Section: Rank68} we treat the elliptic surface of rank $68$, which has previously been seen to be able to be decomposed into $11$ ``smaller'' elliptic surfaces, $10$ of which are rational and can be understood using our algorithm, and one $K3$ elliptic surface.

\subsection{Acknowledgments}
The authors would like to thank John Voight for bringing this subject to their attention, and for their support. The first author is supported by a Postgraduate Research Scholarship in Mathematics and Statistics from the University of Sydney, while the second author is supported by a grant from the Simons Foundation (SFI-MPS-Infrastructure-00008650, JV).

\section{Computing Splitting Fields and Their Automorphism Groups} \label{Section:Steph}

\subsection{Using the GaloisGroup package}

Using \Magma\ one can compute the Galois group of a given polynomial $f \in \bbQ[X]$ by calling {\tt GaloisGroup}.
This returns the automorphism group of the splitting field of $f$ as a permutation group of the roots of $f$.
Further, the roots of $f$ are returned as elements in the complex numbers or in a $p$-adic splitting field.

This information is sufficient to build the splitting field of $f$ as a number field. The simplest way to
do this is to call {\tt GaloisSubgroup}. This will construct a resolvent 
corresponding to a subgroup given. If the trivial subgroup is supplied, a resolvent $g$ defining the splitting field is returned.




\subsection{Results of GaloisSubgroup}

We use the notation $r_1,\ldots,r_n$ for the roots of the initial polynomial $f$, $G \subset S_n$
for the Galois group, and $I$ for the invariant $I$ that was used to construct the resolvent $g$. 
Then one root of $g$ is given by 
$$
-I(r_{1},\ldots,r_{n}).
$$
Further, all the roots of $g$ are given by the 
Galois orbit
$$
-I(r_{1^\sigma},\ldots,r_{n^\sigma}) \mbox{ for all } \sigma \in G\, .
$$

\subsection{Automorphisms of the splitting field}

As the resolvent $g$ defines the normal extension $K := \bbQ[X] / (g)$, one can ask for an explicit description
of its automorphism group. In other words, assume that the root $-I(r_1,\ldots,r_n)$
is presented by $\overline{X} \in K$, one has to construct the other roots of $g$ as elements in $K$. 
As the above roots are $p$-adic approximations of the roots of $g$, one could do this by
rational reconstruction. But here we will describe a more efficient approach.

For each $\sigma \in G$, we are looking for a polynomial $h \in \bbQ[X]$ of degree less than $\deg(g)$ with 
$$
h(-I(r_1,  \ldots,r_n)) = -I(r_{1^\sigma},  \ldots,r_{n^\sigma}) \, .
$$
As $h$ has rational coefficients, we get
$$
h(-I(r_{1^\tau},  \ldots,r_{n^\tau})) = -I(r_{1^{\sigma \tau}},  \ldots,r_{n^{\sigma \tau}})
$$
for all $\tau \in G$. The elements $-I(r_{1^\tau},  \ldots,r_{n^\tau}), \tau \in G$ are 
distinct. Thus, the values of the polynomial $h$ are known at more than $\deg(h)$ points
and $h$ can be constructed by interpolation.

This gives $h$ with any $p$-adic precision. If we assume that $r_1,\ldots,r_n$ are algebraic integers
and $I$ has only integer coefficients, the roots of $g$ are algebraic integers as well. Using \cite[Lemma 10.1.20]{Cohen2}, we have
$$
h(\overline{X}) \in
\calO_K \subset \frac{1}{g^\prime(\overline{X})} \bbZ[\overline{X}] \subset \bbQ[X] / g(X) = K \, .
$$
Thus, one can conclude 
$$
h(\overline{X}) \cdot g^\prime(\overline{X}) \in \bbZ[\overline{X}]\, .
$$
This is used to reconstruct the product above and derive $h(\overline{X}) \in \calO_K$ from it.

\subsection{Using Newton lifting}

In practice, we solve the above interpolation problem only in the residue field (i.e., with $p$-adic precision 1)
giving us $h_0 \in \bbZ[X]$.
To raise the precision, we use Newton lifting
$$
h_{n+1} = h_n - \frac{g(h_n(\overline{X}))}{g^\prime(h_n(\overline{X}))} \, .
$$
As the Newton method converges quadratically, $h_n$ has $p$-adic precision $2^n$. 
For the reconstruction, we form symmetric integer representatives of the coefficients of 
$h_n(\overline{X}) \cdot g^\prime(\overline{X})$ and check if this leads to a 
$g^\prime(\overline{X})$-multiple of an exact root of $g$. If not, we increase the precision by working with a larger value of $n$.

The same strategy was used in~\cite{elsenhans2019computing} to construct embeddings of subfields. The interested reader will find a-priori estimates for the required $p$-adic precision in this article as well.

\subsection{The roots of $f$}
For completeness, we describe the computation of the 
roots of $f$ in $K$. This is done in a 
similar way. For each $i = 1,\ldots,n$ we have to determine $h \in \bbQ[X]$ of degree less than $\deg(g)$ with
$$
r_i = h(\overline{X}) = h(-I(r_1,\ldots,r_n)) \,.
$$
Using the Galois action one more time results in 
$$
r_{i^\tau} = h(-I(r_{1^\tau},\ldots,r_{n^\tau})) \mbox{ for all } \tau \in G \, .
$$
Thus, $h$ can be computed by interpolation.
We denote by $h_0$ a solution of the interpolation problem with $p$-adic
precision 1.
The $p$-adic precision of $h_0$ can be increased by
the Newton iteration
$$
h_{n+1} = h_n - \frac{f(h_n(\overline{X}))}{f^\prime(h_n(\overline{X}))} \,.
$$
The rational reconstruction of the coefficients of $h$ is done in the same way as above. 

\subsection{SplitAutGrp Function}

All of this has been combined into the function \emph{SplitAutGrp}, which runs on \Magma{}.

\begin{algorithm}[H]\label{Algorthim Split}
\caption{\bf SplitAutGrp}
\KwIn{ A sequence of polynomials defined over the rational numbers, whose product is squarefree}
\KwOut{\begin{itemize}
    \item Splitting Field
    \item Galois Group
    \item Mapping from the Galois group to the automorphism group
    \item Roots of the initial polynomials
\end{itemize}}
\end{algorithm}
The algorithm computes each output as described in the previous sections, and the code for this can be found at \cite[AutSplit.m]{RatES-Code}

\begin{rmk}
    The algorithm also requires the choice of a prime number. If no prime is given, the algorithm will choose a prime based on the output of \emph{GaloisGroup}, though this may fail, for instance, if this causes the discriminant of the polynomial corresponding to the output of \emph{GaloisSubgroup} to be zero modulo this prime. A good prime $p$ is one where all of the polynomials involved, in both the inputs and the outputs, have distinct roots modulo $p$. If a bad prime error occurs, the user has to rerun the computation with another prime. As the most time consuming steps of the computation are done after the prime is confirmed to be good, only a small proportion of the run time is spent on bad primes.
\end{rmk}

\begin{example}
Consider the polynomials $\{x^3 - x^2 - 3x + 1, x^2-x-1, x^5 - x^4 - 5x^3 + 4x^2 + 3x - 1 \}$. The function \textit{SplitAutGrp} finds the splitting field of these polynomials to be a number field of degree $120$, the Galois Group $C_2 \times S_3 \times D_5$, and takes less than $20$ seconds.
\end{example} 
\section{Rational Elliptic Surfaces} \label{Section: Rational}
For rational elliptic surfaces, it is known exactly what possibilities exist for the Mordell-Weil group:
\begin{thm}\cite[Table 8.2]{schutt2019mordell}
    For a rational elliptic surface, the Mordell-Weil group $E(\overline{\Q}(t))$ can take the form:
\[ \begin{cases} 
      \Z^r & 0 \leq r \leq 8 \\
      \Z^r \oplus \Z/2\Z & 0 \leq r \leq 4\\
      \Z^r \oplus \Z/3\Z & 0 \leq r \leq 2\\
      \Z^r \oplus (\Z/2\Z)^2 & 0 \leq r \leq 2\\
      \Z^r \oplus \Z/4\Z & 0 \leq r \leq 1\\
   \end{cases}
\]
or one of $\{\Z/5\Z, \Z/6\Z, (\Z/3\Z)^2, \Z/4\Z \oplus \Z/2\Z \}$.
\end{thm}

\begin{proof}[Proof of Theorem \ref{Rational Elliptic SUrface Theorem}]
It is known, for example \cite[Theorem 8.33]{schutt2019mordell}, that for any rational elliptic surface given in the form as in \eqref{Rational Weiestrass Form}, then there are at most $240$ $\overline{\Q}(t)$-integral points, that is, points $P=(x,y)$ of the form \[x= at^2+bt+c, \quad y = dt^3 + et^2 + ft + g, \quad a,\dots,g \in \overline{k_0}.\] Furthermore, these points generate the Mordell-Weil Group $E(\overline{\Q}(t))$. 
Substituting the degree-bounded polynomials into the Weierstrass equation yields a finite system of polynomial equations in the coefficients $a_1, \dots, a_7$. Using this, we get an ideal $I \subset k[a_1,\dots,a_7]$. This ideal $I$ encodes the possible tuples $(a_1,\dots,a_7)$ that yield a valid section. By the Lasker-Noether Theorem, this ideal $I$ admits a decomposition \[I = \mathfrak{q}_1\cap \dots \cap \mathfrak{q}_i\] where each $\mathfrak{q}_i$ is primary, with associated prime $\mathfrak{p}_i =\sqrt{\mathfrak{q}_i}$. Since the degrees of $x$ and $y$ are bounded, the variety $V(I)$ consists of finitely many points, and hence $I$ is zero-dimensional. By Hilbert's Nullstellensatz, each primary component $\mathfrak{q}_i$ corresponds to a distinct family of solutions, of which there are finitely many.
For each primary component $\mathfrak{q}_i$, passing to the associated prime $\mathfrak{p}_i = \sqrt{\mathfrak{q}_i}$ the elimination ideal $\mathfrak{p}_i \cap \Q[a_j] = (f_{j}(a_{j}))$ yields an irreducible polynomial $f_j \in \Q[x]$, which is computable via a lexicographic Gröbner basis. This polynomial is the minimal polynomial of the coefficient $a_j$ over $\Q$, and the field of definition of the corresponding family of sections is $K_i = \Q[x]/(f_j(x))$. Since there are a finite number of these polynomials, and they all have finite degree, their splitting field is a finite extension of $\Q$. We can then recover the sections by computing the roots of each minimal polynomial $f_j$ over $K_i$ and substituting the resulting coefficient values back into the prescribed polynomial form. Since the number of primary components is finite and each yields finitely many solutions, this process terminates with at most $240$ explicit sections. Since we know that the integral points generate the Mordell-Weil group, we can use the standard theory of heights of points to determine the generating points of this set, as in \cite{silverman2009arithmetic} and \cite{silverman2013advanced}, giving the rank, torsion points, and generating sections of the rational elliptic surface. 
Applying each automorphism of the Galois group to the finitely many coefficients of each section is a finite iteration, producing a finite matrix group. Computing the characters of its irreducible representations terminates by standard finite group theory.
This algorithm is explained in more detail in Algorithm \ref{alg:CompCG}.
\end{proof}

A significant bottleneck in the running time is in the height pairing needed to determine independent points of the set of integral points. Instead of working over number fields, we prefer to work over fields of finite characteristic. 

\begin{thm}[{\cite[Proposition 6.2]{van2007elliptic}}]\label{Heron Triangle Injection}
Let $A$ be a discrete valuation ring of a number field $L$ with residue field $k \isom \F_q$ for $q=p^r$. Let $S$ be an integral scheme with a morphism $S \rightarrow \Spec A$ that is projective and smooth of relative dimension $2$. Let $\overline{S} = S_{\overline{\Q}}$ and $\Tilde{S} = S_{\overline{k}}$. There are natural injective homomorphisms $$\mathrm{NS}(\overline{S}) \otimes_{\Z_l} \Q_l \hookrightarrow \mathrm{NS}(\Tilde{S}) \otimes_{\Z_l} \Q_l \hookrightarrow H^2 (\Tilde{S},\Q_l)(1)$$ of finite dimensional vector spaces over $\Q_l$, for $l$ a prime not equal to $p$. The second injection respects the Galois action of $\Galois(\overline{k}/k)$.
\end{thm}

A prime ideal $\mathfrak{p} \subset \mathcal{O}_L$ is of good reduction for the elliptic surface if the discriminant of the elliptic surface does not vanish modulo $\mathfrak{p}$, the reduction of the surface is smooth, and $\mathfrak{p}$ is unramified in $L$. Take $p = char(\mathbb{F}_\mathfrak{p}),$ and let $l\neq p$. By \cite[Theorem 6.6]{schutt2019mordell}, we have an embedding of $E(L(t))\otimes_{\Z}\Q \hookrightarrow \mathrm{NS}(E(\overline{\Q}(t))) \otimes_\Z \Q$.

Since $\mathrm{NS}(E(\overline{\Q}(t)))$ is a finitely generated abelian group, there is a canonical isomorphism \[\mathrm{NS}(E(\overline{\Q}(t))) \otimes_{\mathbb{Z}} \mathbb{\Q}_l \cong \mathrm{NS}(E(\overline{\Q}(t))) \otimes_{\mathbb{Z}_l} \mathbb{\Q}_l,\] obtained by first forming the $l$-adic completion $\mathrm{NS}(E(\overline{\Q}(t))) \otimes_{\mathbb{Z}} \mathbb{Z}_l$ and then inverting $l$. Since $\mathbb{\Q} \subset \mathbb{\Q}_l$, any $\mathbb{\Q}$-linear relation among classes in $\mathrm{NS}(E(\overline{\Q}(t))) \otimes_{\mathbb{Z}} \mathbb{\Q}$ is in particular a $\mathbb{\Q}_l$-linear relation in $\mathrm{NS}(E(\overline{\Q}(t))) \otimes_{\mathbb{Z}_l} \mathbb{\Q}_l$; therefore $\mathbb{\Q}_l$-linear independence implies $\mathbb{\Q}$-linear independence, and it suffices to establish independence in $\mathrm{NS}(E(\overline{\Q}(t))) \otimes_{\mathbb{Z}_l} \mathbb{\Q}_l$.
By Theorem \ref{Heron Triangle Injection}, the natural map \[\mathrm{NS}(E(\overline{\Q}(t))) \otimes_{\mathbb{Z}_l} \mathbb{\Q}_l \longrightarrow \mathrm{NS}(\overline{L}(t)) \otimes_{\mathbb{Z}_l} \mathbb{\Q}_l \] is injective, so it further suffices to verify independence of the corresponding classes in $\mathrm{NS}(E(\overline{L}(t))) \otimes_{\mathbb{Z}_l} \mathbb{\Q}_l$. 
Although the specialization map is injective, it does not typically preserve the height pairing. This is because the height pairing can be given explicitly as \[\langle P,Q \rangle = \chi + (P.O) + (Q.O) - (P.Q) - \sum_{v \in R}contr_v(P,Q),\] where $\chi$ is the Euler characteristic of the surface, $(P.Q)$ represents the intersection number of given sections $P$ and $Q$, and $contr_v(P,Q)$ is the local contribution from the singular reducible fibre above $v$, for $v$ in the finite set $R$ of points on the base curve with reducible fibre. However, integral sections of elliptic surfaces, meet only the identity component of every reducible fibre \cite[Section 8.7]{schutt2019mordell}, hence $contr_v(P,Q)$ is always zero for all $v \in R$ for integral sections. For the remaining terms, at a prime of good reduction the special fibre is smooth, and the reduction map sends divisors to divisors and preserves their intersection multiplicities, so these numbers are unchanged. Therefore, for integral sections at a good prime $\mathfrak{p}$, the height pairing satisfies $\langle P,Q\rangle = \langle \tilde{P}, \tilde{\Q}\rangle$, and so linear independence of the corresponding classes in $\mathrm{NS}(\overline{E}) \otimes_{\Z_l} \Q_l$ can be verified by working over the residue field $\F_{\mathfrak{p}}$, and by the above discussion this confirms the linear independence of the generators of the Mordell-Weil group $E(L(t)).$

We now present the algorithm:

\begin{algorithm}[H]\label{Algorithm Rational}
\label{alg:CompCG}
\caption{\bf Computing Invariants of Rational Elliptic Surfaces}
\KwIn{A globally minimal Weierstrass form of a rational elliptic surface, defined over a number field}
\KwOut{\begin{itemize}
    \item The Rank
    \item The Field of Definition of the Mordell-Weil Group
    \item The Automorphism Group of the above
    \item The generating sections of the Mordell-Weil Group
    \item The torsion points
    \item The irreducible representation
    \item The characters of the irreducible representations over the rationals
\end{itemize}}

\begin{enumerate}
    \item Take sections of the elliptic surface of the form $(a_1t^2+a_2t+a_3,a_4t^3 + a_5t^2 + a_6t + a_7)$, and substitute these into the Weierstrass equation for the elliptic surface, forming an ideal over $k_0[a_1,\dots,a_7]$.
    \item Compute the primary decomposition of this ideal, and extract the polynomials needed to define the number fields.
    \item Compute the splitting field, the Galois group, and the automorphism group of the splitting field. This is done using the function \textit{SplitAutGrp}. 
    \item Generate the (at most) $240$ integral sections of the elliptic surface over the number field.
    \item From these computed integral sections, determine the generating and torsion sections.
    \item Construct the Galois representation by applying the automorphisms to the coefficients of each generating section.
\end{enumerate} 
\end{algorithm}
The code for this can be found at \cite{RatES-Code}.

\subsection{Examples}
We illustrate the algorithm and code with some examples.

\begin{example}
Consider the rational elliptic surface given by $y^2 = x^3 + t^4 + t^3$. The algorithm takes less than $5$ seconds to run on \Magma. This elliptic surface has rank $4$, with field of definition being a number field generated by the polynomial $x^6 + 324x^4 + 186624x^3 + 
    34992x^2 - 60466176x + 8708389056$. The Galois Group of this number field is isomorphic to $S_3$. The Galois representation of the elliptic surface consists of two irreducible degree $2$ representations.
\end{example}

\begin{example}
    Consider the rational elliptic surface given by $y^2 = x^3 -3t(t^2-1)x + (t^2-1)^2$. The algorithm takes about $49$ minutes to run on Magma. This elliptic surface has rank $5$, with field of definition being a number field generated by a polynomial of degree $48$. The Galois Group of this number field is isomorphic to $C_2 \times S_4$. The Galois representation of the elliptic surface consists of three representations, the trivial representation, and $\epsilon \tensor \phi$ and $\epsilon \tensor \rho$, where $\epsilon$ is the nontrivial representation of $C_2$, $\phi$ is the sign representation of $S_4$, and $\rho$ is the tensor product of the standard and sign representation of $S_4$. 
\end{example}

\section{The Elliptic Surface of Rank $68$.}\label{Section: Rank68}
In this section, our base field $k_0$ will always be taken to be $\Q$.
The Lefschetz number is a birational invariant of a surface, and is defined as $$\lambda(X) = b_2(X) - \rho,$$ where $b_2(X)$ is the second Betti number of the surface. Rewriting the Shioda-Tate formula in terms of the Lefschetz number, we get:
$$\text{rk} E(\overline{\mathbb{Q}}(t)) = b_2(X) - \lambda -2 - \sum_{v\in \mathbb{P}^1}(m_v - 1).$$
 In \cite{shioda1986explicit}, an explicit algorithm was given for computing the Lefschetz number, $\lambda(X)$ of Delsarte surfaces. A Delsarte surface $X$ is a surface defined by a sum of $4$ monomials in a projective or affine $3$-space. We present a slightly different formulation of Shioda's original algorithm, specific to elliptic surfaces, given in \cite{heijne2012maximal}:

\begin{thm}[Shioda's Method for Delsarte Surfaces.]\label{Rank Algo}
        Let $f = \sum_{i=0}^3 t^{a_{i0}}X^{a_{i1}}Y^{a_{i2}}Z^{a_{i3}}$ be the homogenized polynomial corresponding to the Elliptic Surface $E$, $A_E = (a_{ij})$ be the matrix consisting of exponents appearing in $f$ and let $L$ be the subgroup of $(\mathbb{Q}/\mathbb{Z})^4$ generated by 
        $$(1,0,0,-1)A_E^{-1}, (0,1,0,-1)A_E^{-1}, \text{ and } (0,0,1,-1)A_E^{-1}.$$ 
Define 
        \begin{multline*}
            \Lambda = \{(a_1,a_2,a_3,a_4)\in L | a_i \neq 0 \forall i \text{ and }  \exists s\in \Z, \\ \text{ such that } \ord(sa_i) = \ord(a_i) \text{ and } \sum_{i=1}^4 \{ sa_i\} \neq 2\},
        \end{multline*}
        where $\ord$ is the order in the additive group $\Q/\Z$ and $\{a_i \}$ is the natural bijection between $[0,1 )\cap \Q$ and $\Q/\Z$.
        Then $\lambda = \#\Lambda$.

\end{thm}

In \cite{shioda1992some}, Shioda presented the elliptic surface $$y^2 = x^3 + t^m + 1.$$
In particular, using the process of Theorem \ref{Rank Algo}, they were able to show that if  $m$ is a multiple of $360$, this elliptic surface has rank $68$ over $\Bar{\Q}(t)$. Most notably, for any elliptic surface, $68$ is currently the rank record over $\Bar{\Q}(t)$, though there are also $5$ other families of elliptic surfaces whose maximal rank is also $68$, which can be seen in \cite{heijne2012maximal}.
We wish to determine the field of definition for $y^2 = x^3 + t^{360} + 1$.

\begin{proof}[Proof of Theorem \ref{Rank 68 Surface Theorem}]
    By Lemma $3.1$ in \cite{chahal1999sections} this elliptic surface has trivial torsion over $\Bar{\Q}(t)$ if $f(x) = x^{360} + 1 $ has a simple root in $\Bar{\Q}$. We can see that this has $360$ distinct roots, and so all of its roots are simple roots.
As a result, we may consider the vector space given by $E(\Bar{\Q}(t))\otimes_{\Z} \Q$,  without needing to consider any torsion points.

Let $K = k(t)$, where $k$ is a number field. For any finite extension $L/K$, the Galois group $\Galois(L/K)$ acts on the group $E(L)$ of $L$-rational points of $E$. Regarding $E(L)$ as a module over $R:= \End_K(E)$, this Galois action is $R$-linear. This can be used to decompose $E(L)$, or rather the vector space $V:=E(L) \otimes_\Z \Q$, as a direct sum of smaller spaces. This was done in \cite{chahal1999sections}, where it was shown that we can consider the $11$ smaller vector spaces given by the following elliptic surfaces:
\begin{center}
    \begin{tabular}{c c}
       $y^2 = x^3 + t^2(t+1)$,  & $y^2 = x^3 + t^3(t+1)$, \\
       $y^2 = x^3 + t(t^2+1)$,  & $y^2 = x^3 + t^2(t^2+1)$,\\
       $y^2 = x^3 + t^3(t^2+1)$, & $y^2 = x^3 + t(t^3+1)$, \\
       $y^2 = x^3 + t^2(t^3+1)$, & $y^2 = x^3 + t^5+1$,\\
       $y^2 = x^3 + t(t^5+1)$, & $y^2 = x^3 + t(t^4+1)$,\\
    \end{tabular}
    
    $y^2 = x^3 + t^5 + t^{-5}$.
\end{center}

The first $10$ of these elliptic surfaces are rational, and so can be understood using Algorithm \ref{Algorithm Rational}, and the results of the computations are summarized in Table \ref{tab:table Rationals Elliptic Surfaces from E68}, where LMFDB label refers to the labeling of all of the needed polynomials whose splitting field is the field of definition, and where $\#\rho$ refers to the dimensions of the irreducible representations in the image of $\rho$. The output of Algorithm \ref{Algorithm Rational} for each of the rational elliptic surfaces can be found \cite[RatES_for_360]{RatES-Code}
\begin{table}[h]
    \centering
    \begin{tabular}{|c|c|c|c|c|c|}
    \hline
      Surface & Rank & $[k:\Q]$ & $\Galois(k/\Q)$ & LMFDB Label(s) & $\#\rho$\\
      \hline
      $y^2 = x^3 + t^2(t+1)$ & 2 & 2 & $C_2$ & \href{https://www.lmfdb.org/NumberField/2.0.3.1}{2.0.3.1}& 1,1\\
      $y^2 = x^3 + t^3(t+1)$ & 2 & 2 & $C_2$ & \href{https://www.lmfdb.org/NumberField/2.0.3.1}{2.0.3.1}& 1,1\\
      $y^2 = x^3 + t(t^2+1)$ & 4 & 16 & $QD_{16}$ &\href{https://www.lmfdb.org/NumberField/8.2.573308928.1}{8.2.573308928.1}  & 4\\
      $y^2 = x^3 + t^2(t^2+1)$ & 4 & 6 & $S_3$ &\href{https://www.lmfdb.org/NumberField/6.0.34992.1}{6.0.34992.1} & $2^2$\\
      $y^2 = x^3 + t^3(t^2+1)$ & 4 & 16 & $QD_{16}$ &\href{https://www.lmfdb.org/NumberField/8.2.573308928.1}{8.2.573308928.1}  & 4\\
      $y^2 = x^3 + t(t^3+1)$ & 6 & 54 & $C_9 \rtimes C_6$ &\href{https://www.lmfdb.org/NumberField/9.1.11019960576.1}{9.1.11019960576.1}  & 6\\
      $y^2 = x^3 + t^2(t^3+1)$ & 6 & 54 & $C_9 \rtimes C_6$ &\href{https://www.lmfdb.org/NumberField/9.1.11019960576.1}{9.1.11019960576.1}  & 6\\
      $y^2 = x^3 + t(t^5 + 1)$ & 8 & 240 & $(C_5 \rtimes C_8) \rtimes S_3$ &\href{https://www.lmfdb.org/NumberField/8.4.56953125.1}{8.4.56953125.1}  & 8\\
       &  &  &  &\href{https://www.lmfdb.org/NumberField/5.1.162000.1}{5.1.162000.1}  & \\
       &  &  &  & \href{https://www.lmfdb.org/NumberField/3.1.300.1}{3.1.300.1} & \\
      $y^2 = x^3 + t^5 + 1$ & 8 & 240 & $(C_5 \rtimes C_8) \rtimes S_3$ &\href{https://www.lmfdb.org/NumberField/8.4.56953125.1}{8.4.56953125.1}  &8\\
       &  &  &  &\href{https://www.lmfdb.org/NumberField/5.1.162000.1}{5.1.162000.1}  &\\
       &  &  &  &\href{https://www.lmfdb.org/NumberField/3.1.300.1}{3.1.300.1}&\\
      $y^2 = x^3 + t(t^4 + 1)$ & 8 & 48 & $C_2 \times D_{12}$ &\href{https://www.lmfdb.org/NumberField/12.0.35664401793024.4}{12.0.35664401793024.4}  & 8\\
      & & & &\href{https://www.lmfdb.org/NumberField/12.0.54780521154084864.58}{12.0.54780521154084864.58}  & \\

      \hline
    \end{tabular}
    \caption{Rational Elliptic Surfaces needed for $E_{68}$.}\label{tab:table Rationals Elliptic Surfaces from E68}
\end{table}

We only need to explore the $K3$ elliptic surface. The field of definition, as well as the $16$ independent points of this elliptic surface have also been given in \cite{salami2022generators}.

\begin{lemma}
    The $K3$ elliptic surface given by $y^2 = x^3 + t^5 + t^{-5}$ has rank $16$, and the Field of Definition of its Mordell-Weil group is a number field of degree $192$.
\end{lemma}

\begin{proof}
    Using Theorem \ref{Rank Algo}, we can show that this elliptic surface has rank $16$. Consider $u = t + t^{-1}$. This gives us a degree $2$ extension $\Bar{\Q}(t)/\Bar{\Q}(u)$ obtained by adjoining a square root of $u^2 - 4$ to $\Bar{\Q}(u)$. By applying \cite[Exercise 10.16]{silverman2009arithmetic}, this reduces to, after a change of variables, considering the two elliptic surfaces given by $$y^2 = x^3 + (t^5 - 5t^3+5t) \quad \text{ and } \quad y^2 = x^3 + (t^2-4)^3(t^5 - 5t^3+5t).$$
Since the first of these is a rational elliptic surface, we can apply Algorithm \ref{Algorithm Rational} to determine its arithmetic information. It has rank $8$, and the degree of its field of definition is $96$, with $240$ integral sections. It is the splitting field of the polynomials 
$x^{16} - 60x^{12} - 370x^8 - 300x^4 + 25$, $x^4 -2x^3+4x^2+12x+6$ and $x^3 - 5$.
We now consider the number field given as the splitting field of the polynomials above, as well as $x^4-x^3+x^2-x+1$. This gives us a number field of degree $192$, which we denote by $k$. We can now lift our integral sections from $y^2 = x^3 + (t^5 - 5t^3+5t)$ to the $K3$, and they have the form 
$$x= au^2+bu+c, \quad y = du^3 + eu^2 + fu + g, \quad a,\dots,g \in k.$$
Instead of attempting to find $8$ independent sections from the elliptic surface $y^2 = x^3 + (t^2-4)^3(t^5 - 5t^3+5t)$, we instead construct new sections on the surface.
On each of these integral sections, we also apply the $k$-automorphism $\sigma$ on $k(t)$ given by $\sigma_i(t) = \zeta_5^i t$, where $\zeta_5$ is a primitive root of $x^5-1$. For each integral section, this gives us a total of $5$ different sections for each integral section from our rational elliptic surface, for a total of $1200$ sections on our $K3$ elliptic surface. Using Tate's Algorithm \cite[IV.9]{silverman2013advanced}, we see that the elliptic surface has no reducible fibers, so we can reduce modulo a prime ideal using Theorem \ref{Heron Triangle Injection}. This tells us that we have $16$ independent generators. In order to verify that this number field of degree $192$ really is the smallest possible number field to have all these points defined, we consider the Galois representation.

The explicit generators $P_1,\dots,P_{16}$ are defined over the field $k$ of degree $192$ constructed above and span $E(\overline{k}(t))\otimes_{\Z}\Q$; hence any $\tau\in Gal(\overline{k}/k)$ fixes each $P_i$ and, by divisibility in the torsion-free group $E(\overline{k}(t))$, acts trivially on all of $E(\overline{k}(t))$. Thus $Gal(\overline{k}/k)\subseteq\ker\rho$, and $\rho$
descends to $\bar\rho\colon Gal(k/\Q)\to Aut\big(E(\overline{k}(t)),\langle\cdot,\cdot\rangle\big)$.
Using \Magma{} we verify $\bar\rho$ is injective, equivalently $ |Im(\rho)|=192=[k:\Q]$, and so $k$ is the field of definition for this elliptic surface. Furthermore, the Galois representation consists of two irreducible components, both of degree $8$. A file containing the code and output for this can be found at \cite[lemma_4_2.m]{RatES-Code}
\end{proof}

We can lift all of our generating sections from their original elliptic surfaces to the elliptic surface of rank $68$. This is done through a combination of base change of the rational elliptic surfaces, and a change of variables. The process is reversing the process to decompose $y^2 = x^3 + t^{360} + 1$ into the $11$ elliptic surfaces, and full details of this can be found in \cite[Chapter 4]{heijne2011elliptic}.
In fact, for the rational elliptic surfaces, we can show exactly how these are lifted in \ref{tab: Lifting Points}.
\begin{table}[h]
    \centering
    \begin{tabular}{|c|c|}
    \hline
      Surface  & Coordinates lifted to $E_{68}$\\
      \hline
      $y^2 = x^3 + t^2(t+1)$ & $(at^{480} +bt^{120} + ct^{-240},dt^{720}+et^{360}+f+gt^{-360})$\\
      $y^2 = x^3 + t^3(t+1)$ & $(at^{360} +b + ct^{-360},dt^{540}+et^{108}+ft^{-108}+gt^{-540})$\\
      $y^2 = x^3 + t(t^2+1)$ & $(at^{300} +bt^{120} + ct^{-60},dt^{450}+et^{270}+ft^{90}+gt^{-90})$\\
      $y^2 = x^3 + t^2(t^2+1)$ & $(at^{120} +bt^{60} + ct^{-120},dt^{360}+et^{180}+f+gt^{-180})$\\
      $y^2 = x^3 + t^3(t^2+1)$ & $(at^{180} +b + ct^{-180},dt^{270}+et^{90}+ft^{-90}+gt^{-270})$\\
      $y^2 = x^3 + t(t^3+1)$ & $(at^{200} +bt^{80} + ct^{-40},dt^{300}+et^{180}+ft^{60}+gt^{-120})$\\
      $y^2 = x^3 + t^2(t^3+1)$ & $(at^{140} +bt^{60} + ct^{-80},dt^{240}+et^{120}+f+gt^{-120})$\\
      $y^2 = x^3 + t(t^5 + 1)$ & $(at^{120} +bt^{48} + ct^{-24},dt^{180}+et^{108}+ft^{36}+gt^{-36})$\\
      $y^2 = x^3 + t^5 + 1$ & $(at^{144} +bt^{72} + c,dt^{216}+et^{144}+ft^{72}+g)$\\
      $y^2 = x^3 + t(t^4 + 1)$ & $(at^{150} +bt^{60} + ct^{-30},dt^{225}+et^{135}+ft^{45}+gt^{-45})$\\
      \hline
    \end{tabular}
    \caption{Points of the form $(at^2+bt+c,dt^3+et^2+ft+g)$ lifted to points on $E_{68}$} \label{tab: Lifting Points}
    
\end{table}
For the $K3$ elliptic surface given by $y^2 = x^3 + t^5 + t^{-5}$, points lifted to the rank $68$ surface have the form $( t^{60}(gu^2+au+b),t^{90}( hu^3 + cu^2 + du + e))$, where $u = t^{36} + t^{-36}$ or $u = \zeta_5t^{36} + \zeta_5^{-1}t^{-36}$, with the coefficients taken from the rational elliptic surface $y^2 = x^3 + t^5 -5t^3 + 5t$. 

Take $H$ to be the subgroup of $E(\overline{\Q}(t))$ corresponding to the Mordell-Weil group of all of the $11$ elliptic surfaces considered. Since $H$ has finite index $n$ in $E(\overline{\Q}(t))$, every element $x$ satisfies
$nx\in H$. In particular a $\Z$-basis $h_1,\dots,h_{68}$ of $H$ is also a $\Q$-basis
of $V$. Hence any $g\in G$ fixing $H$ pointwise fixes this basis, so acts as the
identity on $V$, and therefore on $E(\overline{\Q}(t))$ modulo torsion. The representation
attached to $H$ thus has the same kernel as the one attached to the full
Mordell-Weil group, so by the Galois correspondence the two splitting fields
coincide and the field of definition computed from $H$ is that of $E(\overline{\Q}(t))$.

Combining all of the polynomials needed for the field of definition of the $11$ elliptic surfaces, we see that in order to find the field of definition of the rank $68$ elliptic surface, we need to find the splitting field  of all of the following polynomials:
\[
\left\{
\renewcommand{\arraystretch}{0.8}
\begin{array}{cc}
x^2 - x + 1, & x^8 - 6x^4 - 3, \\[2pt]
x^6 - 3x^5 + 5x^3 - 3x + 1, & x^9 + 6x^3 - 2, \\[2pt]
x^5 - x^4 + 4x^3 + 4x^2 - x + 13, & x^3 - x^2 - 3x - 3,  \\[2pt]
x^{12} - 4x^9 + 22x^6 - 12x^3 + 2, & x^{12} + 4x^6 + 54, \\[2pt]
x^4 - x^3 + x^2 - x + 1, & x^4 - 2x^3 + 4x^2 + 12x + 6, \\[2pt]
x^3 - 5, & x^{16} - 60x^{12} - 370x^8 - 300x^4 + 25, \\[2pt]
\multicolumn{2}{c}{\scriptstyle x^8 - x^7 + 2x^6 + 2x^5 - 5x^4 + 13x^3 - 13x^2 + x + 1,} \\
\end{array}
\right\}
\]
Using \Magma{}, we determine that the splitting field of these polynomials has a Galois group of order $829,440$, and so the size of the splitting field that contains all of these roots, and hence contains all of the $68$ generators, has degree $829,440$.
\end{proof}

\begin{rmk}
    In order to determine the smallest degree polynomial which has the number field given in Theorem $4.2$ as its splitting field, we run the \Magma{} function \textit{MinimalDegreePermutationRepresentation(G)}. The output was a group acting on a set of cardinality $45$, and this computations took about $12.5$ days.
\end{rmk}

\begin{rmk}
    The polynomials given in the proof of Theorem \ref{Rank 68 Surface Theorem} were not the polynomials given as output from Algorithm \ref{Algorithm Rational}, but were instead found by considering the \Magma{} function \textit{GaloisSubgroup}, which finds subfields associated to subgroups of the Galois group. Code can be found at \cite[Field_Matching]{RatES-Code} which verifies that each of the number fields from the computation match the polynomials given in Table \ref{tab:table Rationals Elliptic Surfaces from E68}
.\end{rmk}
\subsection{Future Work}
In \cite{heijne2011elliptic} the author explored the field of definition for the elliptic surface with maximal rank on all of the possible $42$ families of elliptic Delsarte surfaces. They saw that they were able to reduce the problem of finding independent points on these elliptic surfaces into finding independent points on rational and $K3$ elliptic surfaces, though they were unable to give any explanation for this remarkable fact. As well as giving a justification for this, it would be ideal to produce an algorithm which is guaranteed to return the sequence of rational and $K3$ elliptic surfaces which need to be considered for some Delsarte elliptic surface. Finally, some algorithm analogous to Algorithm \ref{Algorithm Rational} for $K3$ surfaces would also be required in order to fully determine the information associated to a Delsarte elliptic surface.
\printbibliography

\begin{thebibliography}{B25}
\bibitem{EK}
Andreas-Stephan Elsenhans, Juergen Klueners:
{\it Computing Subfields of Number Fields and Applications to Galois Group Computations.}
Journal of Symbolic Computation (2018)

\bibitem{Shioda86}
Tetsuji Shioda, \emph{An explicit algorithm for computing the Picard number of certain algebraic surfaces}.

\bibitem{Top99}
MATTHIJS MEIJER, AND JAAP TOP
Matthijs Meijer and Jaap Top, \emph{Section on Certain $j = 0$ Elliptic Surfaces}.

\bibitem{HeijnePhD}
Bas Heijne, \emph{Elliptic Delsarte Surfaces}

\bibitem{Heijne11}
Bas Heijne, \emph{The Maximal Rank of Elliptic Delsarte Surfaces}

\bibitem{PPVW16}
Jennifer Park, Bjorn Poonen, John Voight, and Maleaine Matchett Wood, \emph{A Heutistic for Boundedness of Ranks of Elliptic Curves}.

\bibitem{SS19}
Matthias Schütt and Tetsuji Shioda, \emph{Mordell-Weil Lattices}.

\bibitem{S92}
Tetsuji Shioda, \emph{Some remarks on elliptic curves over function fields}

\bibitem{B23}
S.E. Bootsma, \emph{On Certain Elliptic Surfaces With j-Invariant Zero Over Prime Fields of Positive Characteristic}.

\bibitem{AEC94}
Joseph H. Silverman, \emph{Advanced Topics in the Arithmetic of Elliptic Curves}.

\bibitem{S16}
Joseph H. Silverman, \emph{The Arithmetic of Elliptic Curves}.

\bibitem{L07}
Ronald Martinus van Luijk, \emph{An elliptic K3 surface associated to Heron triangles}

\bibitem{S22}
Sajad Salami, \emph{Generators and splitting fields of certain elliptic K3 surfaces}

\bibitem{S25}
Sajad Salami, \emph{The splitting field and generators of the elliptic surface $Y^2=X^3+t^{360}+1$}

\end{thebibliography}

\end{document}